\documentclass[11pt]{amsart}
\usepackage{amsmath,amssymb}
\usepackage{graphicx}
\usepackage{epstopdf}
\usepackage{amsmath}
\usepackage[active]{srcltx}
\usepackage{t1enc}
\usepackage[latin2]{inputenc}
\usepackage{verbatim}
\usepackage{amsmath,amsfonts,amssymb,amsthm}
\usepackage[mathcal]{eucal}
\usepackage{enumerate}
\usepackage[centertags]{amsmath}

\newtheorem{theorem}{Theorem}

\newtheorem{lemma}{Lemma}

\newtheorem{corollary}{Corollary}

\newtheorem*{W}{Theorem W}
\newtheorem*{GG}{Theorem GG (\cite{GG})}
\newtheorem*{GG1}{Corollary GG1 (\cite{GG})}
\newtheorem*{GG2}{Corollary GG2 (\cite{GG})}

\begin{document}
\author{Ushangi Goginava }
\title[Marcinkiewicz-Fej\'er means ]{Pointwise convergence of
Marcinkiewicz-Fejér means of double Vilenkin-Fourier series}
\address{U. Goginava, Department of Mathematics, Faculty of Exact and
Natural Sciences, Ivane Javakhishvili Tbilisi State University, Chavchavadze
str. 1, Tbilisi 0128, Georgia}
\email{zazagoginava@gmail.com}
\maketitle

\begin{abstract}
In this paper we give a characterization of points in which
Marcinkiewicz-Fej\'er means of double Vilenkin-Fourier series converges.
\end{abstract}

\footnotetext{%
\par
The research was supported by Shota Rustaveli National Science Foundation
grant no.DI/9/5-100/13 (Function spaces, weighted inequalities for integral
operators and problems of summability of Fourier series)}

\section{Introduction}

Lebesgue's \cite{leb} theorem is well known for trigonometric Fourier
series: the Fej{\'e}r means $\sigma_n f$ of $f$ converge to $f$ almost
everywhere if $f\in L_1([0,2\pi ))$ (see also Zygmund \cite{zy}). The
analogous result for Walsh-Fourier series is due to Fine \cite{Fi}. Later
Schipp \cite{Sc} showed that the maximal operator $\sigma ^{*}$ of the Fejér
means of the one-dimensional Walsh-Fourier series is of weak type (1,1),
from which the a. e. convergence follows by standard arguments.

Marcinkievicz \cite{ma} verified for two-dimensional trigonometric Fourier
series that the Mar\-cinkiewicz-Fej{é}r means 
\begin{equation*}
\sigma _{n}\left( f\right) =\frac{1}{n}\sum\limits_{j=1}^{n}S_{j,j}\left(
f\right)
\end{equation*}%
of a function $f\in L\log L([0,2\pi )\times \lbrack 0,2\pi ))$ converge a.e.
to $f$ as $n\rightarrow \infty $, where $S_{j,j}\left( f\right) $ denotes
the cubic partial sums of the Fourier series of $f$. Later Zhizhiashvili 
\cite{zi2,zi1} extended this result to all $f\in L_{1}([0,2\pi )\times
\lbrack 0,2\pi ))$. The analogous result for two-dimensional Walsh-Fourier
series is due to Weisz \cite{We2}.

In the one-dimensional case the set of convergence is characterized with the
help of Lebesgue points. It is known that a.e. point $x\in [0,2\pi )$ is a
Lebesgue point of $f\in L_{1}\left( [0,2\pi )\right) $ and the Fejér means
of the trigonometric Fourier series of $f\in L_{1}\left( [0,2\pi )\right) $
converge to $f$ at each Lebesgue point (see Butzer and Nessel \cite{BN}).
Weisz \cite{We1} introduced the notion of Walsh-Lebesgue points and proved
the analogous results for Walsh-Fourier series.

For Vilenkin-Fourier series the author and Gogoladze \cite{GG} introduced
the notion of Vilenkin-Lebesgue points and proved that Fej\'er means of the
Vilenkin-Fourier series of $f\in L_{1}\left( G_{m}\right) $ converges to $f$
at each Vilenkin-Lebesgue points. In this paper we will generalize these
results for the Marcinkiewicz-Fejér means of double Vilenkin-Fourier series
and characterize the set of convergence of these means. We introduce the
Marcinkiewicz-Lebesgue points and prove that a.e. point is a
Marcinkiewicz-Lebesgue point of an integrable function $f$ and the
Marcinkiewicz-Fejér means of the double Vilenkin-Fourier series of $f$
converge to $f$ at each Marcinkiewicz-Lebesgue point.

The problems of summability of cubical partial sums of multiple Fourier
series have been investigated in (\cite{gat}-\cite{GW}).

\section{Definitions and Notation}

Let $\mathbb{N}_{+}$ denote the set of positive integers, $\mathbb{N}:=%
\mathbb{N}_{+}\cup \{0\}.$ Let $m:=(m_{0},m_{1},...)$ denote a sequence of
positive integers not less than $2.$ Denote by $Z_{m_{k}}:=\{0,1,...,m_{k}-1%
\}$ the additive group of integers modulo $m_{k}$. Define the group $G_{m}$
as the complete direct product of the groups $Z_{m_{j}},$ with the product
of the discrete topologies of $Z_{m_{j}}$'s. The direct product $\mu $ of
the measures 
\begin{equation*}
\mu _{k}(\{j\}):=\frac{1}{m_{k}}\quad (j\in Z_{m_{k}})
\end{equation*}%
is the Haar measure on $G_{m}$ with $\mu (G_{m})=1.$ If the sequence $m$ is
bounded, then $G_{m}$ is called a bounded Vilenkin group. The elements of $%
G_{m}$ can be represented by sequences $x:=(x_{0},x_{1},...,x_{j},...)$, $%
(x_{j}\in Z_{m_{j}}).$ The group operation $+$ in $G_{m}$ is given by $%
x+y=\left( x_{0}+y_{0}\left( \text{mod}m_{0}\right) ,...,x_{k}+y_{k}\left( 
\text{mod}m_{k}\right) ,...\right) $ , where $x=\left(
x_{0},...,x_{k},...\right) $ and $y=\left( y_{0},...,y_{k},...\right) \in
G_{m}$. The inverse of $+$ will be denoted by $-$. In this paper we will
consider only bounded Vilenkin group.

It is easy to give a base for the neighborhoods of $G_{m}:$ 
\begin{equation*}
I_{0}(x):=G_{m},
\end{equation*}%
\begin{equation*}
I_{n}(x):=\{y\in G_{m}|y_{0}=x_{0},...,y_{n-1}=x_{n-1}\}
\end{equation*}%
for $x\in G_{m},\ n\in {\mathbb{N}}$. Define $I_{n}:=I_{n}(0)$ for $n\in {%
\mathbb{N}}_{+}$. Set $e_{n}:=\left( 0,...,0,1,0,...\right) \in G_{m}$ the $%
n\,$th\thinspace coordinate of which is 1 and the rest are zeros $\left(
n\in \mathbb{N}\right) .$

If we define the so-called generalized number system based on $m$ in the
following way: $M_{0}:=1,M_{k+1}:=m_{k}M_{k}(k\in {\mathbb{N}}),$ then every 
$n\in {\mathbb{N}}$ can be uniquely expressed as $n=\sum\limits_{j=0}^{%
\infty }n_{j}M_{j},$ where $n_{j}\in Z_{m_{j}}\ (j\in {\mathbb{N}}_{+})$ and
only a finite number of $n_{j}$'s differ from zero. We use the following
notation. Let (for $n>0$) $|n|:=\max \{k\in {\mathbb{N}}:n_{k}\neq 0\}$
(that is, $M_{|n|}\leq n<M_{|n|+1}$).

Next, we introduce on $G_{m}$ an orthonormal system which is called the
Vilenkin system. At first define the complex valued functions $%
r_{k}(x):G_{m}\rightarrow {\mathbb{C}}$, the generalized Rademacher
functions in this way 
\begin{equation*}
r_{k}(x):=\exp \left( \frac{2\pi \imath x_{k}}{m_{k}}\right) \ (\imath
^{2}=-1,\ x\in G_{m},\ k\in \mathbb{N}).
\end{equation*}

It is known that 
\begin{equation}
\sum\limits_{i=0}^{m_{n}-1}r_{n}^{i}\left( x\right) =\left\{ 
\begin{array}{l}
0,\text{if }x_{n}\neq 0 \\ 
m_{n},\text{ if }x_{n}=0%
\end{array}%
.\right.  \label{rad}
\end{equation}

\noindent Now define the Vilenkin system $\psi :=(\psi _{n}:n\in {\mathbb{N}}%
)$ on $G_{m}$ as follows: 
\begin{equation*}
\psi _{n}(x):=\prod\limits_{k=0}^{\infty }r_{k}^{n_{k}}(x)\quad (n\in 
\mathbb{N}).
\end{equation*}

\noindent Specifically, we call this system the Walsh-Paley one if $m\equiv
2.$

\noindent The Vilenkin system is orthonormal and complete in $L^{1}(G_{m})$ 
\cite{AVDR}.

We consider the double system $\left\{ \psi _{n}(x)\times \psi
_{m}(y):\,n,m\in \mathbb{N}\right\} $ on $G_{m}\times G_{m}$. The notation $%
a\lesssim b$ in the whole paper stands for $a\leq cb$, where $c$ is an
absolute constant.

The rectangular partial sums of the double Vilenkin-Fourier series are
defined as follows:

\begin{equation*}
S_{M,N}\left( x,y;f\right) :=\sum\limits_{i=0}^{M-1}\sum\limits_{j=0}^{N-1}%
\widehat{f}\left( i,j\right) \psi _{i}\left( x\right) \psi _{j}\left(
y\right) ,
\end{equation*}%
where the number 
\begin{equation*}
\widehat{f}\left( i,j\right) =\int\limits_{G_{m}\times G_{m}}f\left(
x,y\right) \psi _{i}\left( x\right) \psi _{j}\left( y\right) d\mu \left(
x,y\right)
\end{equation*}%
is said to be the $\left( i,j\right) $th Vilenkin-Fourier coefficient of the
function\thinspace $f.$

The norm (or quasinorm) of the space $L_{p}\left( G_{m}\times G_{m}\right) $
is defined by 
\begin{equation*}
\left\Vert f\right\Vert _{p}:=\left( \int\limits_{G_{m}\times
G_{m}}\left\vert f\left( x,y\right) \right\vert ^{p}d\mu \left( x,y\right)
\right) ^{1/p}\,\,\,\,\left( 0<p<+\infty \right) .
\end{equation*}%
The space weak-$L_{p}\left( G_{m}\times G_{m}\right) $ consists of all
measurable functions $f$ for which 
\begin{equation*}
\left\Vert f\right\Vert _{\text{weak}-L_{p}\left( G_{m}\times G_{m}\right)
}:=\sup\limits_{\lambda >0}\lambda \mu \left( \left\vert f\right\vert
>\lambda \right) ^{1/p}<+\infty .
\end{equation*}

The $\sigma $-algebra generated by the dyadic 2-dimensional $I_{k}\times
I_{k}$ cube of measure $M_{k}^{-1}\times M_{k}^{-1}$ will be denoted by $%
F_{k}\left( k\in \mathbb{N}\right) .$ Denote by $f=\left( f^{\left( n\right)
},n\in \mathbb{N}\right) $ one-parameter martingales with respect to $\left(
F_{n},n\in \mathbb{N}\right) $ (for details see, e. g. \cite{WeBook}). The
maximal function of a martingale $f$ is defined by 
\begin{equation*}
f^{\ast }=\sup\limits_{n\in \mathbb{N}}\left\vert f^{\left( n\right)
}\right\vert .
\end{equation*}%
In case $f\in L_{1}\left( G_{m}\times G_{m}\right) $, the maximal function
can also be given by 
\begin{equation*}
f^{\ast }\left( x,y\right) =\sup\limits_{n\in \mathbb{N}}\frac{1}{\mu \left(
I_{n}(x)\times I_{n}(y)\right) }\left\vert \int\limits_{I_{n}(x)\times
I_{n}(y)}f\left( t,u\right) d\mu \left( t,u\right) \right\vert ,\,\,
\end{equation*}%
\begin{equation*}
\left( x,y\right) \in G_{m}\times G_{m}.
\end{equation*}

For $0<p<\infty $ the dyadic martingale Hardy space $H_{p}(G\times G)$
consists of all martingales for which

\begin{equation*}
\left\| f\right\| _{H_{p}}:=\left\| f^{*}\right\| _{p}<\infty .
\end{equation*}

If $f\in L_{1}\left( G_{m}\times G_{m}\right) $ then it is easy to show that
the sequence $\left( S_{M_{n},M_{n}}\left( f\right) :n\in \mathbb{N}\right) $
is a martingale. If $f$ is a martingale, that is $f=(f^{\left( 0\right)
},f^{\left( 1\right) },...)$ then the Vilenkin-Fourier coefficients must be
defined in a little bit different way: 
\begin{equation*}
\widehat{f}\left( i,j\right) =\lim\limits_{k\rightarrow \infty
}\int\limits_{G_{m}\times G_{m}}f^{\left( k\right) }\left( x,y\right) \psi
_{i}\left( x\right) \psi _{j}\left( y\right) d\mu \left( x,y\right) .
\end{equation*}%
The Vilenkin-Fourier coefficients of $f\in L_{1}\left( G_{m}\times
G_{m}\right) $ are the same as the ones of the martingale $\left(
S_{M_{n},M_{n}}\left( f\right) :n\in \mathbb{N}\right) $ obtained from $f$.

For $n=1,2,...$ and a martingale $f$ the Marcinkiewicz-Fejér means of order $%
n$ of the 2-dimensional Vilenkin-Fourier series of the function $f$ is given
by 
\begin{equation*}
\sigma _{n}\left( x,y;f\right) =\frac{1}{n}\sum\limits_{j=0}^{n-1}S_{j,j}%
\left( x,y;f\right) .
\end{equation*}%
If 
\begin{equation*}
K_{n}\left( x,y\right) :=\frac{1}{n}\sum\limits_{k=0}^{n-1}D_{k}\left(
x\right) D_{k}\left( y\right)
\end{equation*}%
denotes the 2-dimensional Marcinkiewicz-Fejér kernel of order $n$ then 
\begin{equation}
\sigma _{n}\left( x,y;f\right) =\int\limits_{G_{m}\times G_{m}}f\left(
t,u\right) K_{n}\left( x-t,y-u\right) d\mu \left( t,u\right) .
\label{integral}
\end{equation}

A bounded measurable function $a$ is a p-atom, if there exists a generalized
square $I\times J\in F_{n}\mathbf{,}$ such that

a) $\int\limits_{I\times J}ad\mu =0$;

b) $\left\| a\right\| _{\infty }\leq \mu (I\times J)^{-1/p}$;

c) supp $a\subset I\times J$.

An operator $T$ which maps the set of martingales into the collection of
measurable functions will be called p-quasi-local if there exist a constant $%
C_{p}>0$ such that for every p-atom $a$ 
\begin{equation*}
\int\limits_{G_{m}\times G_{m}\backslash \left( I\times J\right)
}|Ta|^{p}\leq C_{p}<\infty ,
\end{equation*}%
where $I\times J$ is the support of the atom.

\section{Marcinkiewicz-Lebesgue points}

In the one-dimensional case a point $x\in \left( -\infty ,\infty \right) $
is called a Lebesgue point of a function $f$ if 
\begin{equation*}
\lim\limits_{h\rightarrow 0}\frac{1}{h}\int\limits_{0}^{h}|f\left(x+
t\right) -f\left( x\right) |dt=0.
\end{equation*}
It is known that a.e. point $x\in [0,2\pi )$ is a Lebesgue point of $f\in
L_{1}\left( [0,2\pi )\right) $ and that the Fejér means of the trigonometric
Fourier series of $f\in L_{1}\left( [0,2\pi )\right) $ converge to $f$ at
each Lebesgue point (see Butzer and Nessel \cite{BN}). Feichtinger and Weisz 
\cite{FW} extended these results to two-dimensional trigonometric Fourier
series, to arbitrary summability methods and to all $f\in L\left( \log
L\right) ^{+}\left( [0,2\pi )^{2}\right) $.

Weisz introduced the one-dimensional Walsh-Lebesgue point in \cite{We1}: $%
x\in G_{2}$ is a Walsh-Lebesgue point of $f\in L_{1}\left( G_{2}\right) ,$
if 
\begin{equation*}
\lim\limits_{n\rightarrow \infty
}\sum\limits_{k=0}^{n}2^{k}\int\limits_{I_{n}\left( e_{k}\right) }|f\left(
x+t\right) -f\left( x\right) |dt=0.
\end{equation*}%
He proved that a.e. point $x\in G_{2}$ is a Walsh-Lebesgue point of an
integrable function $f$. Moreover, the Fejér means of the Walsh-Fourier
series of $f\in L_{1}\left( G_{2}\right) $ converge to $f$ at each
Walsh-Lebesgue point. The higher dimensional extension of this result can be
found in \cite{wwalsh-lebesgue,GG}.

In \cite{GG} it is characterized the set of convergence of Vilenkin-Fejér
means. We introduced the operator 
\begin{equation*}
W_{A}f\left( x\right)
:=\sum\limits_{s=0}^{A-1}M_{s}\sum\limits_{r_{s}=1}^{m_{s}-1}\int%
\limits_{I_{A}\left( x-r_{s}e_{s}\right) }|f\left( t\right) -f\left(
x\right) |d\mu \left( t\right) .
\end{equation*}

A point $x\in G_{m}$ is a Vilenkin-Lebesgue point of $f\in L_{1}\left(
G_{m}\right) ,$ if 
\begin{equation*}
\lim\limits_{A\rightarrow \infty }W_{A}f\left( x\right) =0.
\end{equation*}

The following are proved in \cite{GG}.

\begin{GG}
Let $f\in L_{1}\left( G_{m}\right) $, where $G_{m}$ is a bounded Vilenkin
group. Then 
\begin{equation*}
\lim\limits_{n\rightarrow \infty }\sigma _{n}f\left( x\right) =f\left(
x\right)
\end{equation*}%
for all Vilenkin-Lebesgue points of $f$.
\end{GG}

\begin{GG1}
Let $f\in L_{1}\left( G_{m}\right) $, where $G_{m}$ is a bounded Vilenkin
group. Then 
\begin{equation*}
\lim\limits_{A\rightarrow \infty }W_{A}f\left( x\right) =0\,\,\,\,\text{%
\thinspace for a. e. }x\in G_{m},
\end{equation*}%
thus a. e. point is a Vilenkin-Lebesgue point of $f$.
\end{GG1}

\begin{GG2}
Let $f\in L_{1}\left( G_{m}\right) ,$where $G_{m}$ is a bounded Vilenkin
group. Then%
\begin{equation*}
\sigma _{n}\left( x;f\right) \rightarrow f\left( x\right) \text{ for a.e. }%
x\in G_{m}.
\end{equation*}
\end{GG2}

For two-dimensional Walsh-Fourier series Weisz \cite{We2} has proved that
for all $f\in L_{1}\left( G\times G\right) $ the Marcinkiewicz-Fej{é}r means 
$\sigma _{n}f$ converge a.e. to $f$ as $n\rightarrow \infty $. 

In \cite{GW} it is characterized the set of convergence of Marcinkiewicz-Fejé%
r means of two-dimensional Walsh-Fourier series.

For two-dimensional Vilenkin-Fourier series Gat \cite{gat} has proved that
for all $f\in L_{1}\left( G_{m}\times G_{m}\right) $ the Marcinkiewicz-Fej{é}%
r means $\sigma _{n}f$ converge a.e. to $f$ as $n\rightarrow \infty $.

In this paper we will characterize the set of convergence of
Marcinkiewicz-Fejér means with respect to bounded Vilenkin system. We
introduce the Marcinkiewicz-Lebesgue points and prove that a.e. point is a
Marcinkiewicz-Lebesgue point of an integrable function $f$ and the
Marcinkiewicz-Fejér means of the two-dimensional Vilenkin-Fourier series of $%
f$ converge to $f$ at each Marcinkiewicz-Lebesgue point.

Set%
\begin{equation*}
W_{j}\left( x,y;f\right)
:=M_{j}^{-1}\sum\limits_{q=0}^{j-1}\sum\limits_{k=q}^{j-1}\sum%
\limits_{u_{q}=1}^{m_{q}-1}M_{q}M_{k}^{2}
\end{equation*}%
\begin{equation*}
\times \int\limits_{I_{k}\times I_{k}\left( u_{q}e_{q}\right) }\left\vert
f\left( x-t,y-u\right) -f\left( x,y\right) \right\vert r_{k+1,j-1}\left(
t,u\right) d\mu \left( t,u\right) 
\end{equation*}%
\begin{equation*}
+M_{j}^{-1}\sum\limits_{q=0}^{j-1}\sum\limits_{k=q}^{j-1}\sum%
\limits_{t_{q}=1}^{m_{q}-1}M_{q}M_{k}^{2}
\end{equation*}%
\begin{equation*}
\times \int\limits_{I_{k}\left( t_{q}e_{q}\right) \times I_{k}}\left\vert
f\left( x-t,y-u\right) -f\left( x,y\right) \right\vert r_{k+1,j-1}\left(
t,u\right) d\mu \left( t,u\right) 
\end{equation*}%
\begin{equation*}
+\sum\limits_{s=0}^{j}\sum\limits_{i=s}^{j}M_{s}M_{i}\sum%
\limits_{u_{s}=1}^{m_{s}-1}
\end{equation*}%
\begin{equation*}
\times \int\limits_{I_{j}\times I_{i}\left( u_{s}e_{s}\right) }\left\vert
f\left( x-t,y-u\right) -f\left( x,y\right) \right\vert d\mu \left(
t,u\right) 
\end{equation*}%
\begin{equation*}
+\sum\limits_{s=0}^{j}\sum\limits_{i=s}^{j}M_{s}M_{i}\sum%
\limits_{t_{s}=1}^{m_{s}-1}
\end{equation*}%
\begin{equation*}
\times \int\limits_{I_{i}\left( t_{s}e_{s}\right) \times I_{j}}\left\vert
f\left( x-t,y-u\right) -f\left( x,y\right) \right\vert d\mu \left(
t,u\right) ,
\end{equation*}%
where%
\begin{equation*}
r_{i,n}\left( x,y\right) :=\prod\limits_{l=i}^{n}\left(
\sum\limits_{s=0}^{m_{l}-1}\psi _{M_{l}}^{s}\left( x+y\right) \right) .
\end{equation*}%
By (\ref{rad}) it is easy to show that%
\begin{equation*}
r_{i,n}\left( x,y\right) =\left\{ 
\begin{array}{l}
m_{i}m_{i+1}\cdots m_{n},x_{j}+y_{j}\left( \text{mod}m_{j}\right)
=0,j=i,i+1,...,n \\ 
0,\text{otherwise}%
\end{array}%
.\right. 
\end{equation*}

A point $\left( x,y\right) \in G_{m}\times G_{m}$ is a
Marcinkiewicz-Lebesgue point (for bounded Vilenkin group) of $f\in
L_{1}\left( G_{m}\times G_{m}\right) ,$ if 
\begin{equation*}
\lim\limits_{n\rightarrow \infty }W_{n}\left( x,y;f\right) =0.
\end{equation*}

Set%
\begin{equation}
V_{n}\left( x,y;f\right)
:=\sum\limits_{q=0}^{n-1}\sum\limits_{k=q}^{n-1}\sum%
\limits_{t_{q}=1}^{m_{q}-1}\frac{M_{q}M_{k}}{m_{k}}  \label{bound}
\end{equation}%
\begin{equation*}
\times \int\limits_{I_{k}\left( t_{q}e_{q}\right) \times I_{k}}f\left(
x-t,y-u\right) \mathbb{I}_{\left\{ t_{r}+u_{r}\left( \text{mod} m_{r}\right)
=0,r=k+1,...,n-1\right\} }\left( t,u\right) d\mu \left( t,u\right) 
\end{equation*}%
\begin{equation*}
+\sum\limits_{q=0}^{n-1}\sum\limits_{k=q}^{n-1}\sum%
\limits_{u_{q}=1}^{m_{q}-1}\frac{M_{q}M_{k}}{m_{k}}
\end{equation*}%
\begin{equation*}
\times \int\limits_{I_{k}\times I_{k}\left( u_{q}e_{q}\right) }f\left(
x-t,y-u\right) \mathbb{I}_{\left\{ t_{r}+u_{r}\left( \text{mod} m_{r}\right)=0,r=k+1,...,n-1\right\}
}\left( t,u\right) d\mu \left( t,u\right) d\mu \left( t,u\right) 
\end{equation*}%
\begin{equation*}
+\sum\limits_{s=0}^{n}\sum\limits_{i=s}^{n}M_{s}M_{i}\sum%
\limits_{u_{s}=1}^{m_{s}-1}
\end{equation*}%
\begin{equation*}
\times \int\limits_{I_{n}\times I_{i}\left( u_{s}e_{s}\right) }f\left(
x-t,y-u\right) d\mu \left( t,u\right) 
\end{equation*}%
\begin{equation*}
+\sum\limits_{s=0}^{n}\sum\limits_{i=s}^{n}M_{s}M_{i}\sum%
\limits_{t_{s}=1}^{m_{s}-1}
\end{equation*}%
\begin{equation*}
\times \int\limits_{I_{i}\left( t_{s}e_{s}\right) \times I_{n}}f\left(
x-t,y-u\right) d\mu \left( t,u\right) 
\end{equation*}%
\begin{equation*}
:=\sum\limits_{j=1}^{4}V_{n}^{\left( j\right) }\left( x,y;f\right) ,
\end{equation*}%
where $\mathbb{I}_{E}$ is characteristic function of the set $E$.

It is easy to see that $W_{n}f\left( x,y\right) \rightarrow 0$ as $%
n\rightarrow \infty $ if and only if 
\begin{equation*}
\lim\limits_{n\rightarrow \infty }V_{n}\left( |f-f(x,y)|\right) \left(
x,y\right) =0.
\end{equation*}

Let 
\begin{equation*}
Vf:=\sup\limits_{n}|V_{n}f|,\qquad V^{\left( i\right)
}f:=\sup\limits_{n}|V_{n}^{\left( i\right) }f|,\,\,i=1,2,3.4.
\end{equation*}

\section{Main Results}

In this paper we prove that the following are true

\begin{theorem}
\label{th1}Let $f\in L_{1}\left( G_{m}\times G_{m}\right) $. Then 
\begin{equation*}
\lim\limits_{n\rightarrow \infty }\sigma _{n}\left( x,y;f\right) =f\left(
x,y\right)
\end{equation*}%
for all Marcinkiewicz-Lebesgue points of $f$.
\end{theorem}

\begin{theorem}
\label{th2}Let $p>1/2$. Then 
\begin{equation*}
\left\Vert Vf\right\Vert _{p}\leq c_{p}\left\Vert f\right\Vert
_{p}\,\,\,\,\,\left( f\in H_{p}\left( G_{m}\times G_{m}\right) \right)
\end{equation*}%
and 
\begin{equation*}
\sup\limits_{\lambda }\lambda \mu \left\{ Vf>\lambda \right\} \leq
c\left\Vert f\right\Vert _{1}.
\end{equation*}
\end{theorem}

It is easy to show that $\lim\limits_{n\rightarrow \infty }W_{n}\left(
x,y;f\right) =0$ for every Vilenkin polynomials and $\left( x,y\right) \in
G_{m}\times G_{m}$. Since the Vilenkin polynomials are dense in $%
L_{1}(G_{m}\times G_{m}),$ Theorem 2 and the usual density argument (see
Marcinkiewicz and Zygmund \cite{MZ}) imply

\begin{corollary}
Let $f\in L_{1}\left( G_{m}\times G_{m}\right) $. Then 
\begin{equation*}
\lim\limits_{n\rightarrow \infty }W_{n}\left( x,y;f\right) =0\,\,\,\,\,\text{%
a. e. }\left( x,y\right) \in G_{m}\times G_{m},
\end{equation*}%
thus a. e. points is a Marcinkiewicz-Lebesgue point of $f$.
\end{corollary}

\begin{corollary}
(Gat \cite{gat}) Let $f\in L_{1}\left( G_{m}\times G_{m}\right) $. Then 
\begin{equation*}
\lim\limits_{n\rightarrow \infty }\sigma _{n}\left( x,y;f\right) =f\left(
x,y\right) \,\,\,\text{a. e. }\left( x,y\right) \in G_{m}\times G_{m}.
\end{equation*}
\end{corollary}

\section{Auxiliary Propositions}

\begin{W}
(Weisz \cite{WeBook}) Suppose that the operator $T$ \thinspace is\thinspace $%
\sigma $-sublinear and $p$-quasi-local for each $0<p_{0}<p\leq 1$ . If $T$
is bounded from $L_{\infty }(G_{m}\times G_{m})$ to $L_{\infty }(G_{m}\times
G_{m})$, then 
\begin{equation*}
\left\Vert Tf\right\Vert _{p}\leq c_{p}\left\Vert f\right\Vert
_{p}\,\,\,\,\,\,\,\,\,\,\,(f\in H_{p}\left( G_{m}\times G_{m}\right) )
\end{equation*}%
for every $0<p_{0}<p<\infty $. In particular for $f\in L_{1}(G_{m}\times
G_{m})$, holds

\begin{equation*}
\left\Vert Tf\right\Vert _{\text{weak}\_L_{1}(G_{m}\times G_{m})}\leq
c\left\Vert f\right\Vert _{1}.
\end{equation*}
\end{W}

\begin{lemma}
\label{lm1}We have%
\begin{equation*}
M_{A}K_{M_{A}}\left( x,y\right) =\sum\limits_{k=0}^{A-1}r_{k+1,A-1}\left(
x,y\right) M_{k}\sum\limits_{r=1}^{m_{k}-1}\left(
\sum\limits_{q=0}^{r-1}\psi _{M_{k}}^{q}\left( x\right) \right) \left(
\sum\limits_{s=0}^{r-1}\psi _{M_{k}}^{s}\left( y\right) \right)
\end{equation*}%
\begin{equation*}
\times D_{M_{k}}\left( x\right) D_{M_{k}}\left( y\right)
\end{equation*}%
\begin{equation*}
+\sum\limits_{k=0}^{A-1}r_{k+1,A-1}\left( x,y\right)
\sum\limits_{r=1}^{m_{k}-1}\left( \sum\limits_{q=0}^{r-1}\psi
_{M_{k}}^{q}\left( x\right) \right) \psi _{M_{k}}^{r}\left( y\right)
D_{M_{k}}\left( x\right) M_{k}K_{M_{k}}\left( y\right)
\end{equation*}%
\begin{equation*}
+\sum\limits_{k=0}^{A-1}r_{k+1,A-1}\left( x,y\right)
\sum\limits_{r=1}^{m_{k}-1}\left( \sum\limits_{s=0}^{r-1}\psi
_{M_{k}}^{s}\left( y\right) \right) \psi _{M_{k}}^{r}\left( x\right)
D_{M_{k}}\left( y\right) M_{k}K_{M_{k}}\left( x\right)
\end{equation*}%
\begin{equation*}
+r_{1,A-1}\left( x+y\right) .
\end{equation*}
\end{lemma}

\begin{proof}[Proof of Lemma \protect\ref{lm1}]
Since (see \cite{GES}) 
\begin{equation*}
D_{j+rM_{A}}\left( x\right) =\left( \sum\limits_{q=0}^{r-1}\psi
_{M_{A}}^{q}\left( x\right) \right) D_{M_{A}}\left( x\right) +\psi
_{M_{A}}^{r}\left( x\right) D_{j}\left( x\right)
\end{equation*}%
\begin{equation*}
j=0,1,...,M_{A}-1,r=1,2,...,m_{A-1}-1,
\end{equation*}%
we can write%
\begin{equation*}
M_{A}K_{M_{A}}\left( x,y\right) =\sum\limits_{j=0}^{M_{A}-1}D_{j}\left(
x\right) D_{j}\left( y\right)
\end{equation*}%
\begin{equation*}
=M_{A-1}K_{M_{A-1}}\left( x,y\right)
+\sum\limits_{r=1}^{m_{A-1}-1}\sum\limits_{j=0}^{M_{A-1}-1}D_{j+rM_{A-1}}%
\left( x\right) D_{j+rM_{A-1}}\left( y\right)
\end{equation*}%
\begin{equation*}
=M_{A-1}K_{M_{A-1}}\left( x,y\right)
+\sum\limits_{r=1}^{m_{A-1}-1}\sum\limits_{j=0}^{M_{A-1}-1}\left(
\sum\limits_{q=0}^{r-1}\psi _{M_{A-1}}^{q}\left( x\right) \right)
D_{M_{A-1}}\left( x\right)
\end{equation*}%
\begin{equation*}
\times \left( \sum\limits_{q=0}^{r-1}\psi _{M_{A-1}}^{q}\left( y\right)
\right) D_{M_{A-1}}\left( y\right)
\end{equation*}%
\begin{equation*}
+\sum\limits_{r=1}^{m_{A-1}-1}\sum\limits_{j=0}^{M_{A-1}-1}\psi
_{M_{A-1}}^{r}\left( x\right) \psi _{M_{A-1}}^{r}\left( y\right) D_{j}\left(
x\right) D_{j}\left( y\right)
\end{equation*}%
\begin{equation*}
+\sum\limits_{r=1}^{m_{A-1}-1}\sum\limits_{j=0}^{M_{A-1}-1}\left(
\sum\limits_{q=0}^{r-1}\psi _{M_{A-1}}^{q}\left( x\right) \right)
D_{M_{A-1}}\left( x\right)
\end{equation*}%
\begin{equation*}
\times \psi _{M_{A-1}}^{r}\left( y\right) D_{j}\left( y\right)
\end{equation*}%
\begin{equation*}
+\sum\limits_{r=1}^{m_{A-1}-1}\sum\limits_{j=0}^{M_{A-1}-1}\left(
\sum\limits_{s=0}^{r-1}\psi _{M_{A-1}}^{s}\left( y\right) \right)
D_{M_{A-1}}\left( y\right)
\end{equation*}%
\begin{equation*}
\times \psi _{M_{A-1}}^{r}\left( x\right) D_{j}\left( x\right)
\end{equation*}%
\begin{equation*}
=\left( \sum\limits_{r=0}^{m_{A-1}-1}\psi _{M_{A-1}}^{r}\left( x+y\right)
\right) M_{A-1}K_{M_{A-1}}\left( x,y\right)
\end{equation*}%
\begin{equation*}
+M_{A-1}\sum\limits_{r=1}^{m_{A-1}-1}\left( \sum\limits_{q=0}^{r-1}\psi
_{M_{A-1}}^{q}\left( x\right) \right) \left( \sum\limits_{s=0}^{r-1}\psi
_{M_{A-1}}^{s}\left( y\right) \right) D_{M_{A-1}}\left( x\right)
D_{M_{A-1}}\left( y\right)
\end{equation*}%
\begin{equation*}
+\sum\limits_{r=1}^{m_{A-1}-1}\left( \sum\limits_{q=0}^{r-1}\psi
_{M_{A-1}}^{q}\left( x\right) \right) \psi _{M_{A-1}}^{r}\left( y\right)
D_{M_{A-1}}\left( x\right) M_{A-1}K_{M_{A-1}}\left( y\right)
\end{equation*}%
\begin{equation*}
+\sum\limits_{r=1}^{m_{A-1}-1}\left( \sum\limits_{s=0}^{r-1}\psi
_{M_{A-1}}^{s}\left( y\right) \right) \psi _{M_{A-1}}^{r}\left( x\right)
D_{M_{A-1}}\left( y\right) M_{A-1}K_{M_{A-1}}\left( x\right) .
\end{equation*}

Iterating this equality we obtain the proof of Lemma \ref{lm1}.
\end{proof}

By results in \cite{PSim} we have%
\begin{equation}
\left\vert K_{M_{A}}\left( x\right) \right\vert \lesssim
\sum\limits_{s=0}^{A}\frac{M_{s}}{M_{A}}\sum%
\limits_{x_{s}=1}^{m_{s}-1}D_{M_{A}}\left( x-x_{s}e_{s}\right)  \label{est1}
\end{equation}%
and%
\begin{equation}
n\left\vert K_{n}\left( x\right) \right\vert \lesssim
\sum\limits_{j=0}^{A}M_{j}\left\vert K_{M_{j}}\left( x\right) \right\vert
,M_{A}\leq n<M_{A+1}.  \label{est2}
\end{equation}%
Then from (\ref{est1}) and (\ref{est2}) we can write%
\begin{eqnarray}
n\left\vert K_{n}\left( x\right) \right\vert &\lesssim
&\sum\limits_{j=0}^{A}\sum\limits_{s=0}^{j}M_{s}\sum%
\limits_{x_{s}=1}^{m_{s}-1}D_{M_{j}}\left( x-x_{s}e_{s}\right)  \label{fejer}
\\
&\lesssim
&\sum\limits_{s=0}^{A}M_{s}\sum\limits_{j=s}^{A}\sum%
\limits_{x_{s}=1}^{m_{s}-1}D_{M_{j}}\left( x-x_{s}e_{s}\right) .  \notag
\end{eqnarray}

\begin{lemma}
\label{lm2}Let $M_{A}\leq n<M_{A+1}$. Then wee have%
\begin{equation*}
n\left\vert K_{n}\left( x,y\right) \right\vert \lesssim
\sum\limits_{j=0}^{A}\sum\limits_{q=0}^{j-1}\sum%
\limits_{k=q}^{j-1}r_{k+1,j-1}\left( x,y\right) 
\end{equation*}%
\begin{equation*}
\times M_{q}D_{M_{k}}\left( x\right)
\sum\limits_{y_{q}=1}^{m_{q}-1}D_{M_{k}}\left( y-y_{q}e_{q}\right) 
\end{equation*}%
\begin{equation*}
+\sum\limits_{j=0}^{A}\sum\limits_{q=0}^{j-1}\sum%
\limits_{k=q}^{j-1}r_{k+1,j-1}\left( x,y\right) 
\end{equation*}%
\begin{equation*}
\times M_{q}D_{M_{k}}\left( y\right)
\sum\limits_{x_{q}=1}^{m_{q}-1}D_{M_{k}}\left( x-x_{q}e_{q}\right) 
\end{equation*}%
\begin{equation*}
+\sum\limits_{j=0}^{A}D_{M_{j}}\left( x\right)
\sum\limits_{s=0}^{j}M_{s}\sum\limits_{i=s}^{j}\sum%
\limits_{y_{s}=1}^{m_{s}-1}D_{M_{i}}\left( y-y_{s}e_{s}\right) 
\end{equation*}%
\begin{equation*}
+\sum\limits_{j=0}^{A}D_{M_{j}}\left( y\right)
\sum\limits_{s=0}^{j}M_{s}\sum\limits_{i=s}^{j}\sum%
\limits_{x_{s}=1}^{m_{s}-1}D_{M_{i}}\left( x-x_{s}e_{s}\right) .
\end{equation*}
\end{lemma}

\begin{proof}[proof of Lemma \protect\ref{lm2}]
It is proved in \cite{GoSMH} that%
\begin{equation*}
n\left\vert K_{n}\left( x,y\right) \right\vert \lesssim
\sum\limits_{j=0}^{A}M_{j}\left\vert K_{M_{j}}\left( x,y\right) \right\vert
\end{equation*}%
\begin{equation*}
+\sum\limits_{j=0}^{A}D_{M_{j}}\left( x\right) \max\limits_{1\leq n\leq
n^{\left( j\right) }}n\left\vert K_{n}\left( y\right) \right\vert
\end{equation*}%
\begin{equation*}
+\sum\limits_{j=0}^{A}D_{M_{j}}\left( y\right) \max\limits_{1\leq n\leq
n^{\left( j\right) }}n\left\vert K_{n}\left( x\right) \right\vert ,
\end{equation*}%
where $n^{\left( j\right) }:=\sum\limits_{k=0}^{j}n_{k}M_{k}$.

Then from Lemma \ref{lm1} and estimation (\ref{est2}) we can write%
\begin{equation*}
n\left\vert K_{n}\left( x,y\right) \right\vert \lesssim
\sum\limits_{j=0}^{A}\sum\limits_{k=0}^{j-1}r_{k+,j-1}\left( x,y\right)
D_{M_{k}}\left( x\right) M_{k}\left\vert K_{M_{k}}\left( y\right)
\right\vert 
\end{equation*}%
\begin{equation*}
+\sum\limits_{j=0}^{A}\sum\limits_{k=0}^{j-1}r_{k+1,j-1}\left( x,y\right)
D_{M_{k}}\left( y\right) M_{k}\left\vert K_{M_{k}}\left( x\right)
\right\vert 
\end{equation*}%
\begin{equation*}
+\sum\limits_{j=0}^{A}D_{M_{j}}\left( x\right) \max\limits_{1\leq n\leq
n^{\left( j\right) }}n\left\vert K_{n}\left( y\right) \right\vert 
\end{equation*}%
\begin{equation*}
+\sum\limits_{j=0}^{A}D_{M_{j}}\left( y\right) \max\limits_{1\leq n\leq
n^{\left( j\right) }}n\left\vert K_{n}\left( x\right) \right\vert 
\end{equation*}%
\begin{equation*}
\lesssim \sum\limits_{j=0}^{A}\sum\limits_{k=0}^{j-1}r_{k+1,j-1}\left(
x,y\right) D_{M_{k}}\left( x\right) 
\end{equation*}%
\begin{equation*}
\times
\sum\limits_{q=0}^{k}M_{q}\sum\limits_{y_{q}=1}^{m_{q}-1}D_{M_{k}}\left(
y-y_{q}e_{q}\right) 
\end{equation*}%
\begin{equation*}
+\sum\limits_{j=0}^{A}\sum\limits_{k=0}^{j-1}r_{k+1,j-1}\left( x,y\right)
D_{M_{k}}\left( y\right) 
\end{equation*}%
\begin{equation*}
\times
\sum\limits_{q=0}^{k}M_{q}\sum\limits_{x_{q}=1}^{m_{q}-1}D_{M_{k}}\left(
x-x_{q}e_{q}\right) 
\end{equation*}%
\begin{equation*}
+\sum\limits_{j=0}^{A}D_{M_{j}}\left( x\right)
\sum\limits_{s=0}^{j}M_{s}\sum\limits_{i=s}^{j}\sum%
\limits_{y_{s}=1}^{m_{s}-1}D_{M_{i}}\left( y-y_{s}e_{s}\right) 
\end{equation*}%
\begin{equation*}
+\sum\limits_{j=0}^{A}D_{M_{j}}\left( y\right)
\sum\limits_{s=0}^{j}M_{s}\sum\limits_{i=s}^{j}\sum%
\limits_{x_{s}=1}^{m_{s}-1}D_{M_{i}}\left( x-x_{s}e_{s}\right) .
\end{equation*}

Lemma \ref{lm2} is proved.
\end{proof}

\section{Proofs of Main Results}

\begin{proof}[Proof of Theorem 1]
We can write 
\begin{equation*}
\left\vert \sigma _{n}\left( x,y;f\right) -f\left( x,y\right) \right\vert 
\end{equation*}%
\begin{equation*}
\leq \int\limits_{G_{m}\times G_{m}}\left\vert f\left( x-t,y-u\right)
-f\left( x,y\right) \right\vert \left\vert K_{n}\left( t,u\right)
\right\vert d\mu \left( t,u\right) 
\end{equation*}%
\begin{eqnarray*}
&\leq &\frac{c}{n}\sum_{j=0}^{A}\sum_{q=0}^{j-1}\sum_{k=q}^{j-1}\sum%
\limits_{u_{q}=1}^{m_{q}-1}\int\limits_{G_{m}\times G_{m}}\left\vert f\left(
x-t,y-u\right) -f\left( x,y\right) \right\vert  \\
&&\times r_{k+1,j-1}(t,u)M_{q}D_{M_{k}}(t)D_{M_{k}}(u-u_{q}e_{q})d\mu \left(
t,u\right) 
\end{eqnarray*}%
\begin{eqnarray*}
&&+\frac{c}{n}\sum_{j=0}^{A}\sum_{q=0}^{j-1}\sum_{k=q}^{j-1}\sum%
\limits_{t_{q}=1}^{m_{q}-1}\int\limits_{G_{m}\times G_{m}}\left\vert f\left(
x-t,y-u\right) -f\left( x,y\right) \right\vert  \\
&&\times r_{k+1,j-1}(t,u)M_{q}D_{M_{k}}(t-t_{q}e_{q})D_{M_{k}}(u)d\mu \left(
t,u\right) 
\end{eqnarray*}%
\begin{eqnarray*}
&&+\frac{1}{n}\sum_{j=0}^{A}\sum_{s=0}^{j-1}\sum_{i=s}^{j}\sum%
\limits_{u_{s}=1}^{m_{s}-1}M_{s}\int\limits_{G_{m}\times G_{m}}\left\vert
f\left( x-t,y-u\right) -f\left( x,y\right) \right\vert  \\
\times  &&D_{M_{j}}\left( t\right) D_{M_{i}}\left( u-u_{s}e_{s}\right) d\mu
\left( t,u\right) 
\end{eqnarray*}%
\begin{eqnarray*}
&&+\frac{c}{n}\sum_{j=0}^{A}\sum_{s=0}^{j-1}\sum_{i=s}^{j}\sum%
\limits_{t_{s}=1}^{m_{s}-1}M_{s}\int\limits_{G_{m}\times G_{m}}\left\vert
f\left( x-t,y-u\right) -f\left( x,y\right) \right\vert  \\
\times  &&D_{M_{j}}\left( t-t_{s}e_{s}\right) D_{M_{i}}\left( u\right) d\mu
\left( t,u\right) 
\end{eqnarray*}

\begin{eqnarray*}
&=&\frac{c}{n}\sum_{j=0}^{A}\sum_{q=0}^{j-1}\sum_{k=q}^{j-1}\sum%
\limits_{u_{q}=1}^{m_{q}-1}M_{q}M_{k}^{2}\int\limits_{I_{k}\times
I_{k}\left( u_{q}e_{q}\right) }\left\vert f\left( x-t,y-u\right) -f\left(
x,y\right) \right\vert  \\
&&\times r_{k+1,j-1}(t,u)d\mu \left( t,u\right) 
\end{eqnarray*}%
\begin{eqnarray*}
&&+\frac{c}{n}\sum_{j=0}^{A}\sum_{q=0}^{j-1}\sum_{k=q}^{j-1}\sum%
\limits_{u_{q}=1}^{m_{q}-1}M_{q}M_{k}^{2}\int\limits_{I_{k}\left(
t_{q}e_{q}\right) \times I_{k}}\left\vert f\left( x-t,y-u\right) -f\left(
x,y\right) \right\vert  \\
&&\times r_{k+1,j-1}(t,u)d\mu \left( t,u\right) 
\end{eqnarray*}%
\begin{equation*}
+\frac{c}{n}\sum_{j=0}^{A}\sum_{s=0}^{j}\sum_{i=s}^{j}\sum%
\limits_{u_{s}=1}^{m_{s}-1}M_{q}M_{i}M_{j}
\end{equation*}%
\begin{equation*}
\times \int\limits_{I_{j}\times I_{i}\left( u_{s}e_{s}\right) }\left\vert
f\left( x-t,y-u\right) -f\left( x,y\right) \right\vert d\mu \left(
t,u\right) 
\end{equation*}%
\begin{equation*}
+\frac{c}{n}\sum_{j=0}^{A}\sum_{s=0}^{j}\sum_{i=s}^{j}\sum%
\limits_{t_{s}=1}^{m_{s}-1}M_{q}M_{i}M_{j}
\end{equation*}%
\begin{equation*}
\times \int\limits_{I_{i}\left( t_{s}e_{s}\right) \times I_{j}}\left\vert
f\left( x-t,y-u\right) -f\left( x,y\right) \right\vert d\mu \left(
t,u\right) 
\end{equation*}%
\begin{equation*}
\leq \frac{c}{n}\sum_{j=0}^{A}M_{j}W_{j}\left( x,y;f\right) 
\end{equation*}%
which tends to $0$ as $n\rightarrow \infty $. This completes the proof of
Theorem \ref{th1}.
\end{proof}

\begin{proof}[Proof of Theorem 2]
Since 
\begin{equation*}
Vf\leq \sum\limits_{j=1}^{4}V^{\left( j\right) }f,
\end{equation*}%
by Theorem W, the proof of Theorem \ref{th2} will be complete if we show
that the operators $V^{\left( i\right) }$ ,$\,i=1,2,3,4\,\,$ are
p-quasi-local for each $1/2<p\leq 1$ and bounded from $L_{\infty
}(G_{m}\times G_{m})$ to $L_{\infty }(G_{m}\times G_{m})$.

It follows from (\ref{bound}) that 
\begin{equation*}
\left\Vert Vf\right\Vert _{\infty }\leq c\left\Vert f\right\Vert _{\infty
}\sup\limits_{n}\sum\limits_{q=0}^{n}\sum\limits_{k=q}^{n}\frac{M_{q}M_{k}}{%
M_{k}M_{n}}\leq c\left\Vert f\right\Vert _{\infty }.
\end{equation*}

Let $a$ be an arbitrary atom with support $I_{N}\left( z^{\prime },z^{\prime
\prime }\right) =I_{N}(z^{\prime })\times I_{N}(z^{\prime \prime })$. It is
easy to see that $V_{n}(a)=0$ if $n<N$. Therefore we can suppose that $n\geq
N$. We may assume that $z^{\prime }=z^{\prime \prime }=0.$ Hence 
\begin{equation}
\text{supp}\left( a\right) \subset I_{N}\times I_{N}.  \label{atom}
\end{equation}

\textbf{Step 1: Integrating over} $\overline{I}_{N}\times \overline{I}_{N}$.
If $q\geq N$ then $y-u\notin I_{N}$ by (\ref{bound}). Hence 
\begin{equation*}
a\left( x-t,y-u\right) =0.
\end{equation*}%
Consequently, we can write 
\begin{equation*}
V_{n}^{\left( 1\right) }\left( x,y;a\right)
:=\sum\limits_{q=0}^{N-1}\sum\limits_{k=q}^{N-1}\frac{M_{k}M_{q}}{m_{k}}%
\sum\limits_{t_{q}=1}^{m_{q}-1}
\end{equation*}%
\begin{equation*}
\int\limits_{I_{k}\left( t_{q}e_{q}\right) \times I_{k}}a\left(
x-t,y-u\right) \mathbb{I}_{\left\{ t_{r}+u_{r}\left( \text{mod} m_{r}\right)
=0,r=k+1,...,n-1\right\} }\left( t,u\right) d\mu \left( t,u\right) .
\end{equation*}%
Then from (\ref{atom}) $a\left( x-t,y-u\right) \neq 0$ implies that 
\begin{equation*}
t=\left( 0,...,0,t_{q},0,...,0,x_{k},...,x_{N-1},t_{N},...\right) 
\end{equation*}%
\begin{equation*}
x=\left( 0,...,0,t_{q},0,...,0,x_{k},...,x_{N-1},...\right) 
\end{equation*}%
\begin{equation*}
u=\left( 0,...,0,y_{k},m_{k+1}-x_{k+1},...,m_{N-1}-x_{N-1},u_{N},...\right) 
\end{equation*}%
\begin{equation*}
y=\left( 0,...,0,y_{k},m_{k+1}-x_{k+1},...,m_{N-1}-x_{N-1},y_{N},...\right) .
\end{equation*}%
Hence 
\begin{eqnarray*}
\left\vert V_{n}^{\left( 1\right) }\left( x,y;a\right) \right\vert 
&\lessapprox &\frac{M_{N}^{2/p}}{M_{N}^{2}}\sum\limits_{q=0}^{N-1}\sum%
\limits_{k=q}^{N-1}M_{k}M_{q}\mathbb{I}_{I_{k}\left( t_{q}e_{q}\right)
}\left( x\right)  \\
&&\times \mathbb{I}_{I_{N}\left(
0,...,0,y_{k},m_{k+1}-x_{k+1},...,m_{N-1}-x_{N-1}\right) }\left( y\right) 
\end{eqnarray*}%
and 
\begin{eqnarray}
\int\limits_{\overline{I}_{N}\times \overline{I}_{N}}\left( V^{\left(
1\right) }\left( x,y;a\right) \right) ^{p}d\mu \left( x,y\right)  &\leq &%
\frac{c_{p}M_{N}^{2}}{M_{N}^{2p}}\sum\limits_{q=0}^{N-1}\sum%
\limits_{k=q}^{N-1}M_{k}^{p}M_{q}^{p}\frac{1}{M_{k}}\frac{1}{M_{N}}
\label{v1-1} \\
&\leq &c_{p}\frac{M_{N}}{M_{N}^{2p}}\sum\limits_{q=0}^{N-1}M_{q}^{p}\sum%
\limits_{k=q}^{N-1}M_{k}^{p-1}  \notag \\
&\leq &c_{p}<\infty \qquad \,\,\left( 1/2<p\leq 1\right) .  \notag
\end{eqnarray}

Analogously, we can prove that%
\begin{equation}
\int\limits_{\overline{I}_{N}\times \overline{I}_{N}}\left( V^{\left(
2\right) }\left( x,y;a\right) \right) ^{p}d\mu \left( x,y\right) \leq
c_{p}<\infty \qquad \,\,\left( 1/2<p\leq 1\right) .  \label{V2-1}
\end{equation}

Since $x+t\notin I_{N}$ we obtain, 
\begin{equation}
V_{n}^{\left( 3\right) }\left( x,y;a\right) =0.  \label{V3-1}
\end{equation}

Analogously, we can prove that 
\begin{equation}
V_{n}^{\left( 4\right) }\left( x,y;a\right) =0.  \label{V4-1}
\end{equation}

Combining (\ref{bound})$\,$and (\ref{v1-1}-\ref{V4-1}) $\,\,$we obtain 
\begin{equation}
\int\limits_{\overline{I}_{N}\times \overline{I}_{N}}\left( V\left(
x,y;a\right) \right) ^{p}d\mu \left( x,y\right) \leq c_{p}\,\,\,\left(
1/2<p\leq 1\right) .  \label{main1}
\end{equation}

\textbf{Step 2: Integrating over} $\overline{I}_{N}\times I_{N}$. Since $%
V_{n}^{\left( 1\right) }\left( x,y;a\right) =0$ for $q\geq N$ we have 
\begin{eqnarray}
&&V_{n}^{\left( 1\right) }\left( x,y;a\right)   \label{V11+V12} \\
&=&\sum\limits_{q=0}^{N-1}\sum\limits_{k=q}^{N-1}\frac{M_{k}M_{q}}{m_{k}}%
\sum\limits_{t_{q}=1}^{m_{q}-1}  \notag \\
&&\int\limits_{I_{k}\left( t_{q}e_{q}\right) \times I_{k}}a\left(
x-t,y-u\right) \mathbb{I}_{\left\{ t_{r}+u_{r}\left( \text{mod} m_{r}\right)
=0,r=k+1,...,n-1\right\} }\left( t,u\right) d\mu \left( t,u\right)   \notag
\\
&&+\sum\limits_{q=0}^{N-1}\sum\limits_{k=N}^{n}\frac{M_{k}M_{q}}{m_{k}}%
\sum\limits_{t_{q}=1}^{m_{q}-1}  \notag \\
&&\int\limits_{I_{k}\left( t_{q}e_{q}\right) \times I_{k}}a\left(
x-t,y-u\right) \mathbb{I}_{\left\{ t_{r}+u_{r}\left( \text{mod} m_{r}\right)
=0,r=k+1,...,n-1\right\} }\left( t,u\right) d\mu \left( t,u\right)   \notag
\\
&=&V_{n}^{\left( 1,1\right) }\left( x,y;a\right) +V_{n}^{\left( 1,2\right)
}\left( x,y;a\right) .  \notag
\end{eqnarray}%
From (\ref{atom}) $V_{n}^{\left( 1,1\right) }\left( x,y;a\right) \neq 0$
implies that 
\begin{equation*}
u=\left( 0,...,0,u_{N},...\right) 
\end{equation*}%
\begin{equation*}
y=\left( 0,...,0,y_{N},...\right) .
\end{equation*}%
\begin{equation*}
t=\left( 0,...,0,t_{q},0,...,0,t_{k},0,...,0,t_{N},...\right) 
\end{equation*}%
\begin{equation*}
x=\left( 0,...,0,t_{q},0,...,0,t_{k},0,...,0,x_{N},...\right) 
\end{equation*}%
Consequently, 
\begin{eqnarray*}
\left\vert V_{n}^{\left( 1,1\right) }\left( x,y;a\right) \right\vert 
&\lesssim &\frac{M_{N}^{2/p}}{M_{N}^{2}}\sum\limits_{q=0}^{N-1}\sum%
\limits_{k=q}^{N-1}M_{k}M_{q}\sum\limits_{t_{q}=1}^{m_{q}-1}\sum%
\limits_{t_{k}=1}^{m_{k}-1} \\
&&\mathbb{I}_{I_{N}\left( e_{q}t_{q}+e_{k}t_{k}\right) }\left( x\right) 
\mathbb{I}_{I_{N}}\left( y\right) 
\end{eqnarray*}%
and 
\begin{equation}
\int\limits_{\overline{I}_{N}\times I_{N}}\left( \sup\limits_{n}\left\vert
V_{n}^{\left( 1,1\right) }\left( x,y;a\right) \right\vert \right) ^{p}\leq 
\frac{c_{p}M_{N}^{2}}{M_{N}^{2p}}\sum\limits_{q=0}^{N-1}\sum%
\limits_{k=q}^{N-1}M_{k}^{p}M_{q}^{p}\frac{1}{M_{N}^{2}}\leq c_{p}.
\label{V11}
\end{equation}

From (\ref{atom})$\,V_{n}^{\left( 1,2\right) }\left( x,y;a\right) \neq 0$
implies that 
\begin{equation*}
t=\left( 0,...,0,t_{q},0,...,0,t_{k},...,t_{n-1},t_{n},...\right)
\end{equation*}%
\begin{equation*}
x=\left( 0,...,0,t_{q},0,...,0,x_{N},...\right)
\end{equation*}%
\begin{equation*}
u=\left( 0,...,0,u_{k},\alpha _{k+1},...,\alpha _{n-1},u_{n},...\right)
\end{equation*}%
\begin{equation*}
y=\left( 0,...,0,y_{N},...\right) ,
\end{equation*}%
where%
\begin{equation*}
\alpha _{j}:=\left\{ 
\begin{array}{c}
m_{j}-t_{j},t_{j}\neq 0 \\ 
0,t_{j}=0,j=k+1,...,n-1.%
\end{array}%
\right.
\end{equation*}%
Thus 
\begin{eqnarray*}
\left\vert V_{n}^{\left( 1,2\right) }\left( x,y;a\right) \right\vert
&\lesssim &\frac{M_{N}^{2/p}}{M_{N}M_{n}}\sum\limits_{q=0}^{N-1}\sum%
\limits_{k=q}^{N-1}M_{k}M_{q}\sum\limits_{t_{q}=1}^{m_{q}-1}\mathbb{I}%
_{I_{N}\left( e_{q}t_{q}\right) }\left( x\right) \mathbb{I}_{I_{N}}\left(
y\right) \\
&\lesssim &\frac{cM_{N}^{2/p}}{M_{n}}\sum\limits_{q=0}^{N-1}M_{q}\mathbb{I}%
_{I_{N}\left( e_{q}t_{q}\right) }\left( x\right) \mathbb{I}_{I_{N}}\left(
y\right)
\end{eqnarray*}%
and $\left( n\geq N\right) $ 
\begin{equation}
\int\limits_{\overline{I}_{N}\times I_{N}}\left( \sup\limits_{n}\left\vert
V_{n}^{\left( 1,2\right) }\left( x,y;a\right) \right\vert \right) ^{p}d\mu
\left( x,y\right) \leq \frac{c_{p}M_{N}^{2}}{M_{N}^{2}M_{N}^{p}}%
\sum\limits_{q=0}^{N-1}M_{q}^{p}\leq c_{p}<\infty .  \label{V12}
\end{equation}%
Combining (\ref{V11+V12})-(\ref{V12}) we conclude that 
\begin{equation}
\int\limits_{\overline{I}_{N}\times I_{N}}\left( \left\vert V^{\left(
1\right) }\left( x,y;a\right) \right\vert \right) ^{p}d\mu \left( x,y\right)
\leq c_{p}<\infty .  \label{V1(2)}
\end{equation}

Let $q<N$. Then it is easy to show that $y-u\notin I_{N}$ and consequently,%
\begin{equation*}
a\left( x-t,y-u\right) =0,
\end{equation*}%
\begin{equation*}
V_{n}^{\left( 2\right) }\left( x,y;a\right) =0.
\end{equation*}%
Let $q\geq N$. Then $x-t\notin I_{N}$ and%
\begin{equation*}
V_{n}^{\left( 2\right) }\left( x,y;a\right) =0.
\end{equation*}

Hence%
\begin{equation}
V^{\left( 2\right) }\left( x,y;a\right) =0.  \label{V2(2)}
\end{equation}

Analogously, we can prove that 
\begin{equation}
V^{\left( 4\right) }\left( x,y;a\right) =0.  \label{V4(2)}
\end{equation}

The estimation of $V_{n}^{\left( 3\right) }\left( x,y;a\right) $ is
analogous to the estimation of $V_{n}^{\left( 1\right) }\left( x,y;a\right) $
and we can prove that

\begin{equation}
\int\limits_{\overline{I}_{N}\times I_{N}}\left( \left\vert V^{\left(
3\right) }\left( x,y;a\right) \right\vert \right) ^{p}d\mu \left( x,y\right)
\leq c_{p}<\infty \text{ }\left( 1/2<p\leq 1\right) .  \label{V3(2)}
\end{equation}

Combining (\ref{V1(2)})-(\ref{V3(2)}) we conclude that%
\begin{equation}
\int\limits_{\overline{I}_{N}\times I_{N}}\left( \left\vert V\left(
x,y;a\right) \right\vert \right) ^{p}d\mu \left( x,y\right) \leq
c_{p}<\infty \text{ }\left( 1/2<p\leq 1\right) .  \label{main2}
\end{equation}

\textbf{Step 3: Integrating over}$\,\,\,I_{N}\times \overline{I}_{N}$. This
case is analogous to the step 2 and we obtain that 
\begin{equation}
\int\limits_{I_{N}\times \overline{I}_{N}}\left( \left\vert V\left(
x,y;a\right) \right\vert \right) ^{p}d\mu \left( x,y\right) \leq
c_{p}<\infty \text{ }\left( 1/2<p\leq 1\right) .  \label{main3}
\end{equation}

Combining (\ref{main1}), (\ref{main2}) and (\ref{main3}) we complete the
proof of Theorem \ref{th2}.
\end{proof}

\end{document}